\documentclass[a4paper,10pt]{article}

\usepackage{geometry}
\usepackage{epsfig}
\usepackage{amsmath}
\usepackage{amssymb}
\usepackage{amsthm}
\usepackage{mathrsfs}
\usepackage{color}
\usepackage{amscd}

\usepackage{mathtools} 					
\usepackage{hyperref}						
 \usepackage{booktabs}					

\usepackage{amsthm}
\newtheorem{theorem}{Theorem}[section]

\newtheorem{lemma}[theorem]{Lemma}
\newtheorem{proposition}[theorem]{Proposition}
\newtheorem{corollary}[theorem]{Corollary}
\newtheorem{remark}[theorem]{Remark}
\newtheorem{example}[theorem]{Example}

\theoremstyle{remark}{
}
\theoremstyle{definition}{
\newtheorem{definition}[theorem]{Definition}}

\usepackage{amsfonts}

\newcommand{\hh}{{\mathbb{H}}}

\newcommand{\rr}{{\mathbb{R}}}

\newcommand{\ext}{\textnormal{ext}}

\newcommand{\s}{{\mathbb{S}}}

\newcommand\im{\operatorname{Im}}

\newcommand{\hslashslash}{%
  \raisebox{.9ex}{%
    \scalebox{.7}{%
      \rotatebox[origin=c]{18}{$-$}%
    }%
  }%
}
\newcommand{\fslash}{%
  {%
   \vphantom{f}%
   \ooalign{\kern.05em\smash{\hslashslash}\hidewidth\cr$f$\cr}%
   \kern.05em
  }%
}

\title{\bf A local representation formula for quaternionic slice regular functions}

\author{Graziano Gentili, Caterina Stoppato\\
\\
\small Dipartimento di Matematica e Informatica ``U. Dini'', Universit\`a degli Studi di Firenze \\
\small Viale Morgagni 67/A, I-50134 Firenze, Italy\\
\small graziano.gentili@unifi.it, caterina.stoppato@unifi.it}

\date{  }

\begin{document}

\maketitle


\begin{abstract}
After their introduction in 2006, quaternionic slice regular functions have mostly been studied over domains that are symmetric with respect to the real axis. This choice was motivated by some foundational results published in 2009, such as the Representation Formula for axially symmetric domains.

The present work studies slice regular functions over domains that are not axially symmetric, partly correcting the hypotheses of some previously published results. In particular, this work includes a Local Representation Formula valid without the symmetry hypothesis. Moreover, it determines a class of domains, called simple, having the following property: every slice regular function on a simple domain can be uniquely extended to the symmetric completion of its domain.
\end{abstract}


\thanks{\small \noindent{\bf Acknowledgements.} This work was partly supported by INdAM, through: GNSAGA; INdAM project ``Hypercomplex function theory and applications''. It was also partly supported by MIUR, through the projects: Finanziamento Premiale FOE 2014 ``Splines for accUrate NumeRics: adaptIve models for Simulation Environments''; PRIN 2017 ``Real and complex manifolds: topology, geometry and holomorphic dynamics''. The authors of this work are deeply grateful to the anonymous reviewer for carefully reading the manuscript and providing useful suggestions.}


\section{Introduction}\label{sec:introduction}

The theory of quaternionic slice regular functions was introduced in~\cite{cras,advances} as a possible quaternionic analog of the theory of holomorphic complex functions. Let us denote the algebra of quaternions as $\hh$; the real axis as $\rr$; the $2$-sphere of quaternionic imaginary units as $\s$; and the $2$-plane spanned by $1$ and by any $I\in\s$ as $L_I$. If $T \subseteq \hh$, for each $I \in \s$, let $T_I := T \cap L_I$. As usual, we endow $\hh=\rr^4$ with the Euclidean topology and call an open connected subset of $\hh$ a \emph{domain}.

\begin{definition}
Let $\Omega\subseteq\hh$ be a domain and consider a function $f:\Omega\to\hh$. For each $I \in \s$, let $f_I := f_{|_{\Omega_I}}$ be the restriction of $f$ to $\Omega_I$. The restriction $f_I$ is called \emph{holomorphic} if it has continuous partial derivatives and
\begin{equation}
\bar \partial_I f(x+yI) := \frac{1}{2} \left( \frac{\partial}{\partial x} + I \frac{\partial}{\partial y} \right) f_I(x+yI) \equiv 0.
\end{equation}
The function $f$ is called \emph{slice regular} if, for all $I \in \s$, $f_I$ is holomorphic.
\end{definition}

While the first works within the theory of slice regular functions focused on the case when $\Omega$ is a Euclidean ball centered at the origin, it soon became clear that there were larger classes of quaternionic domains over which the theory was interesting and had useful applications. Indeed, the next definition and theorem were published in~\cite{advancesrevised} and in~\cite{poli}, respectively.

\begin{definition}\label{slicedomain}
Let $\Omega$ be a domain in $\hh$ that intersects the real axis. $\Omega$ is called a \emph{slice domain} if, for all $I \in \s$, the intersection $\Omega_I$ with the complex plane $L_I$ is a domain of $L_I$.
\end{definition}

\begin{theorem}[Identity Principle]\label{identity}
Let $f,g$ be slice regular functions on a slice domain $\Omega$. If, for some $I \in \s$, $f$ and $g$ coincide on a subset of $\Omega_I$ having an accumulation point in $\Omega_I$, then $f = g$ in $\Omega$.
\end{theorem}

Furthermore, for slice regular functions over slice domains fulfilling the next definition, the work~\cite{advancesrevised} proved a strong property called Representation Formula.

\begin{definition}\label{axiallysymmetric}
A set $T \subseteq \hh$ is called \emph{(axially) symmetric} if, for all points $x+yI \in T$ with $x,y \in \rr$ and $I \in \s$, the set $T$ contains the whole sphere $x+y\s$.
\end{definition}

\begin{theorem}[Representation Formula]\label{R-representationformula}
Let $f$ be a slice regular function on a symmetric slice domain $\Omega$ and let $x+y\s \subset \Omega$. For all $I,J,K \in \s$ with $J \neq K$
\begin{eqnarray}
f(x+yI) &=& (J-K)^{-1} \left[J f(x+yJ) - K f(x+yK)\right] +\label{generalrepresentationformula} \\
&+& I (J-K)^{-1} \left[f(x+yJ) - f(x+yK)\right]\,.\nonumber
\end{eqnarray}
Moreover, the quaternion $b := (J-K)^{-1} \left[J f(x+yJ) - K f(x+yK)\right]$ and the quaternion $c := (J-K)^{-1} \left[f(x+yJ) - f(x+yK)\right]$ do not depend on $J,K$ but only on $x,y$.
\end{theorem}

An alternative proof of the same result is included in the proof of~\cite[Theorem 2.4]{global}. As a consequence of the Representation Formula, every slice regular function on a symmetric slice domain is real analytic, see~\cite[Proposition 7]{perotti}. Another result proven in~\cite{advancesrevised} allows to construct slice regular functions on symmetric slice domains, starting from $\hh$-valued holomorphic functions.

\begin{lemma}[Extension Lemma]\label{extensionlemma}
Let $\Omega$ be a symmetric slice domain and let $I \in \s$. If $f_I : \Omega_I \to \hh$ is holomorphic then there exists a unique slice regular function $g : \Omega \to \hh$ such that $g_I = f_I$ in $\Omega_I$. The function $g$ is denoted by $\ext(f_I)$ and called the \emph{regular extension} of $f_I$.
\end{lemma}

After the work~\cite{advancesrevised}, the development of the theory led to many interesting results valid over symmetric slice domains. These results are collected in the monograph~\cite{librospringer} and in a variety of subsequent works. On the other hand, the study of slice regular functions over slice domains that are not symmetric has not been further developed for several years. A possible reason is that~\cite[Theorem 4.1]{advancesrevised} stated that every slice regular function on a slice domain $\Omega$ could be extended in a unique fashion to the symmetric completion $\widetilde \Omega$ of $\Omega$, in accordance with the next definition.

\begin{definition}
The \emph{(axially) symmetric completion} of a set  $T \subseteq \hh$ is the smallest symmetric set $\widetilde{T}$ that contains $T$. In other words,
\begin{equation}
\widetilde{T} := \bigcup_{x+yI \in T} (x+y\s).
\end{equation}
\end{definition}

However, the recent work~\cite{douren1} disproved~\cite[Theorem 4.1]{advancesrevised} by means of a counterexample. It also pointed out that the proof proposed in~\cite{advancesrevised} implicitly applied the Identity Principle~\ref{identity} over the intersection of two sets, which was not a slice domain. 
The same work~\cite{douren1} then introduced the notions of \emph{Riemann slice domain} and \emph{(Riemann) slice domain of regularity}, both being abstract topological spaces. The Representation Formula~\ref{R-representationformula} has been generalized to (Riemann) slice domains of regularity in~\cite{douren1} and to Riemann slice domains in~\cite{douren2}. The work~\cite{douren2} also explored the algebraic structure of slice regular functions over Riemann slice domains.

In contrast with the approach of~\cite{douren1,douren2}, the present work furthers the study over slice domains of $\hh$ that are not symmetric. Our main result is a local version of Theorem~\ref{R-representationformula} valid without the symmetry hypothesis.

\begin{theorem}[Local Representation Formula]
Let $\Omega$ be a slice domain and let $f : \Omega \to \hh$ be a slice regular function. For all $J,K \in \s$ with $J \neq K$ and all $x,y \in \rr$ with $y > 0$ such that $x+yJ, x+yK \in \Omega$, let us set
\begin{eqnarray*}
b(x+yJ,x+yK) &:=& (J-K)^{-1} \left[J f(x+yJ) - K f(x+yK)\right]\,,\\
c(x+yJ,x+yK) &:=& (J-K)^{-1} \left[f(x+yJ) - f(x+yK)\right] \,.
\end{eqnarray*}
Additionally, let us set $b(x,x) := f(x)$ and $c(x,x) := 0$ for all $x \in \Omega \cap \rr$. For every $p_0 \in \Omega$, there exists a slice domain $\Lambda$ with $p_0 \in \Lambda \subseteq \Omega$ such that the following properties hold for all $x,y \in \rr$ with $y \geq 0$:
\begin{itemize}
\item If $U := (x+y\s) \cap \Lambda$ is not empty, then $b,c$ are constant in $(U \times U) \setminus \{ (u,u) : u \in U \setminus \rr\}$.
\item If $I,J,K \in \s$ with $J \neq K$ are such that $x+yI,x+yJ,x+yK \in \Lambda$, then
\begin{equation}
f(x+yI) = b(x+yJ,x+yK) + I c(x+yJ,x+yK) \,.
\end{equation}
\end{itemize}
\end{theorem}

Moreover, this work discloses an extension phenomenon for the class of domains described in the next definition.  For all $J\in\s$ and all $T\subseteq\hh$, let us use the notations
\[L_J^+:=\{x+yJ : x,y\in\rr,y>0\}\]
and $T_J^+:=T\cap L_J^+$.

\begin{definition}
A slice domain $\Omega$ is \emph{simple} if, for any choice of $J,K \in \s$, the set
\[\Omega_{J,K}^+:=\{x+yJ \in \Omega_{J}^+: x+yK \in \Omega_{K}^+\}\]
is connected.
\end{definition}

Our new result for simple slice domains is the following.

\begin{theorem}[Extension]
Let $f$ be a slice regular function on a simple slice domain $\Omega$. There exists a unique slice regular function $\tilde f : \widetilde{\Omega} \to \hh$ that extends $f$ to the symmetric completion of its domain.
\end{theorem}

Section~\ref{sec:extensionformula} states and proves a slightly corrected version of the General Extension Formula~\cite[Theorem 4.2]{advancesrevised}, needed in the subsequent pages. Section~\ref{sec:representationformula} presents a Local Extension Theorem, which, besides its independent interest, is used to prove the aforementioned Local Representation Formula. Section~\ref{sec:extension} presents a broad class of examples of simple domains, while the domain used in the counterexample of~\cite{douren1} is shown not to be simple. Section~\ref{sec:extension} also includes a proof of the aforementioned Extension Theorem.


\section{The General Extension Formula}\label{sec:extensionformula}

The Extension Formula, published in~\cite[Theorem 4.2]{advancesrevised}, provides a method to construct slice regular functions starting from couples of holomorphic functions. We present here a slightly corrected version of the same result.

\begin{theorem}[Extension Formula]\label{extensionformulathm}
Let $J,K$ be distinct imaginary units; let $T$ be a domain in $L_J$, such that $T_J^+$ is connected and $T \cap \rr \neq \emptyset$; let $U := \{x+yK: x+yJ \in T\}$. Choose holomorphic functions $r : T \to \hh, s : U\to \hh$ such that $r_{|_{T \cap \rr}} = s_{|_{U \cap \rr}}$. Let $\Omega$ be the symmetric slice domain such that $\Omega_J^+=T_J^+, \Omega\cap\rr=T\cap\rr$ and set, for all $x+yI \in \Omega$ with $x,y\in\rr, y\geq0$ and $I\in\s$,
\begin{eqnarray}
f(x+yI) &:=& (J-K)^{-1} \left[J r(x+yJ) - K s(x+yK)\right] + \label{extensionformula}\\
&+& I (J-K)^{-1} \left[r(x+yJ) - s(x+yK)\right] \nonumber
\end{eqnarray}
The function $f: \Omega \to \hh$ is the (unique) slice regular function on $\Omega$ that coincides with $r$ in $\Omega_J^+$, with $s$ in $\Omega_K^+$ and with both $r$ and $s$ in $\Omega\cap\rr$.
\end{theorem}

\begin{proof}
Formula~\eqref{extensionformula} yields
\begin{eqnarray*}
f(x+yI) &=& [(J-K)^{-1}J+I(J-K)^{-1}] r(x+yJ) +\\
&-& [(J-K)^{-1}K+I(J-K)^{-1}] s(x+yK).
\end{eqnarray*}
Since
\[(J-K)^{-1}J + J (J-K)^{-1} = |J-K|^{-2} [(K-J) J + J(K-J)] =\]
\[= [(J-K)(K-J)]^{-1} (2 + JK + KJ) = 1\]
and 
\[(J-K)^{-1}K + J (J-K)^{-1} = |J-K|^{-2} [(K-J) K + J(K-J)] =\]
\[= |J-K|^{-2} (-1 -JK +JK +1) = 0,\]
the function $f$ coincides with $r$ in $\Omega_J^+$ and in $\Omega\cap\rr$. Similarly, $f$ coincides with $s$ in $\Omega_K^+$ and in $\Omega\cap\rr$.

We can prove that $f$ is slice regular in $\Omega\setminus\rr$ by showing that, for each $I\in\s$, $f_I$ is holomorphic in $\Omega_I^+$. Indeed, a direct computation shows that
\begin{eqnarray*}
\bar \partial_I f(x+yI) &=& [(J-K)^{-1}J+I(J-K)^{-1}] \bar \partial_J r(x+yJ) +\\
&-& [(J-K)^{-1}K+I(J-K)^{-1}] \bar \partial_K s(x+yK)
\end{eqnarray*}
for $x+yI\in\Omega_I^+$. Since $\bar \partial_J r(x+yJ) \equiv 0$ in $\Omega_J^+$ and $\bar \partial_K r(x+yK) \equiv 0$ in $\Omega_K^+$, it follows that $\bar \partial_I f(x+yI) \equiv 0$ in $\Omega_I^+$.

Proving that $f$ is slice regular near every point of $\Omega\cap\rr$ requires a bit more work. Each connected component of $\Omega\cap\rr$ admits an open neighborhood $D$ in $\Omega$ that is a symmetric slice domain and whose slice $D_J$ is included in the domain $T$ of the function $r$. The domain $U$ of $s$ automatically includes $D_K$. Let us consider the regular extension $g=\ext(r_{|_{D_J}})$ on $D$: $g$ coincides with the regular extension $\ext(s_{|_{D_K}})$ by the Identity Principle~\ref{identity}, because $r$ and $s$ coincide in $D\cap\rr$. In the symmetric slice domain $D$, we can apply the Representation Formula~\eqref{generalrepresentationformula} to $g$. Taking into account that $g_J=r_{|_{D_J}}$ and $g_K=s_{|_{D_K}}$, we get
\begin{eqnarray*}
g(x+yI) &=& (J-K)^{-1} \left[J r(x+yJ) - K s(x+yK)\right] +\\
&+& I (J-K)^{-1} \left[r(x+yJ) - s(x+yK)\right] \nonumber
\end{eqnarray*}
for all $x+yI\in\Omega$. Thus, $g$ coincides with $f_{|_D}$. In particular, $f$ is slice regular in $D$, as desired.

Our final remark concerns the uniqueness of $f$. The Identity Principle~\ref{identity} proves that, for any slice regular function $h:\Omega\to\hh$ coinciding with $r$ in $\Omega\cap\rr$, $h$ coincides with $f$.
\end{proof}

The original statement of the Extension Formula set $\Omega:=\widetilde{T}$ and used formula~\eqref{extensionformula} for all $x+yI \in \Omega$, without requiring $y\geq0$. However, the following example proves that such an $f$ may be ill-defined.

\begin{example}\label{ex:douren}
Let $\Omega$ be the slice domain constructed in~\cite[page 5]{douren1}: it holds $\Omega\supset\rr$; moreover, for every $J\in\s$, 
\[\Omega_J^+:=L_J^+\setminus(h_J\cup a_J)\]
where
\[h_J:=(-\infty,-2)+2J=\{t+2J : t\in (-\infty,-2)\}\]
is a half line with origin $-2+2J$ and $a_J$ is an appropriately defined arc, with endpoints $-2+2J$ and $2J$, within the closed disk 
\[D_J:=\{z\in L_J : |z+1-2J|\leq1\}\,.\]
The definition of $a_J$ for $J\in\s$ is the following: an imaginary unit $I$ is fixed; for each $J\in\s$, it is set $T(J):=\min\{|J-I|,1\}$ and
\[a_J:=-1+2J+\{(1-T(J))e^{2\pi J t}+T(J) e^{-2\pi J t}: t\in[0,1/2]\}\,.\]
In particular, $a_I$ is the upper half of the circle $\partial D_I$ within $L_I^+$ and, for every imaginary unit $J$ with $|J-I|\geq1$ (including $J=-I)$, $a_J$ is the lower half of the circle $\partial D_J$ within $L_J^+$.

Consider the planar domain
\[T:=\Omega_I=L_I \setminus( h_I \cup h_{-I} \cup a_I \cup a_{-I} )\]
and its conjugate
\[U:=\overline{\Omega_I}=L_I \setminus ( h_I \cup h_{-I}\cup \overline{a_{-I}}  \cup \overline{a_I} ).\]
Consider the unique holomorphic functions $r:T\to L_I$ and $s:U\to L_I$ such that $r(x+2I)=\ln(x)=s(x+2I)$ for all $x\in(0,+\infty)$. The intersection $T\cap U$ has three connected components: namely, the open disks that form the interiors of $D_I$ and $D_{-I}$ in $L_I$ and the set $L_I\setminus(h_I \cup h_{-I} \cup D_I \cup D_{-I})$. In the second and third connected component, the functions $r$ and $s$ coincide; in the first component, they differ by a jump. If we apply Theorem~\ref{extensionformulathm} to $r$ and $s$ with $J=I$ and $K=-I$, we get a slice regular function
\[\hh\setminus(\widetilde{h_I}\cup\widetilde{a_I})\to\hh.\]
If, instead, we apply Theorem~\ref{extensionformulathm} to $r$ and $s$ with $J=-I$ and $K=I$, we get a slice regular function
\[\hh\setminus(\widetilde{h_I}\cup\widetilde{a_{-I}})\to\hh.\]
These slice regular functions coincide in the symmetric slice domain $\hh\setminus(\widetilde{h_I}\cup\widetilde{D_I})$ but they differ by a jump in the interior of $\widetilde{D_I}\setminus D_{-I}$.
\end{example}

The function $r$ used in the previous example is the restriction $G_I:\Omega_I\to\hh$ of the slice regular function $G:\Omega\to\hh$ constructed in~\cite[page 5]{douren1}.


\section{Local extension and representation over slice domains}\label{sec:representationformula}

This section is devoted to proving the Local Extension Theorem and the Local Representation Formula announced in the Introduction. We begin with a useful lemma. In the statement, the expression ``closed interval'' includes the degenerate interval consisting of a single point but excludes the empty set.

\begin{lemma}\label{gamma}
Let $Y$ be an open subset of $\hh$ and let $J_0\in\s$. Let $C$ be a compact and path-connected subset of $Y_{J_0}$ such that $C \cap \rr$ is a closed interval and $\emptyset \neq C \setminus \rr \subset Y_{J_0}^+$. Let $q_0 \in C$ be such that $\max_{p \in C} |\im(p)| = |\im(q_0)|$. Then there exists $\varepsilon>0$ such that
\[\Gamma(C,\varepsilon) := \bigcup_{p \in C \setminus \rr} B\left(p,\frac{|\im(p)|}{|\im(q_0)|}\varepsilon\right) \cup \bigcup_{p \in C \cap \rr} B(p,\varepsilon)\]
is a slice domain and $C \subset \Gamma(C,\varepsilon) \subseteq Y$.
\end{lemma}

\begin{proof}
By compactness, there exists an $\varepsilon>0$ such that, for every $p \in C$, the Euclidean ball $B(p,\varepsilon)$ is included in $Y$. Up to shrinking $\varepsilon$, it also holds $\varepsilon<|\im(q_0)|$. Clearly, $\Gamma = \Gamma(C,\varepsilon)$ has the desired property $C \subset \Gamma \subseteq Y$. Let us prove that $\Gamma$ is a slice domain.

The union of open balls, each centered at one point of $C$, is a domain: it is obviously open; it is path-connected because any point can be joined to the center of a ball by a line segment, while centers are connected by paths within $C$.

The domain $\Gamma$ intersects $\rr$ by construction.

Moreover, for each $I\in\s$, the slice $\Gamma_I$ is a domain in $L_I$. Indeed, let $\vartheta_0:=\arcsin\frac{\varepsilon}{|\im(q_0)|}$ and let $\vartheta \in [0,\pi/2]$ be the angle between $L_I$ and $L_{J_0}$ within the $3$-space $\rr+I\rr+J_0\rr$. There are two possibilities.
\begin{itemize}
\item If $\vartheta\geq\vartheta_0$ then, for any $p \in C \setminus \rr$, the distance $|\im(p)|\sin\vartheta$ between $p$ and $L_I$ is greater than, or equal to, $|\im(p)| \sin\vartheta_0=\frac{|\im(p)|}{|\im(q_0)|}\varepsilon$. Thus,
\[\Gamma_I = \bigcup_{p \in C \cap \rr} B(p,\varepsilon) \cap L_I\]
is a union of open disks centered at points of the closed interval $C \cap \rr$.
\item If $\vartheta<\vartheta_0$ then every ball $B\left(p,\frac{|\im(p)|}{|\im(q_0)|}\varepsilon\right)$ with $p \in C \setminus \rr$ and every ball $B(p,\varepsilon)$ with $p \in C \cap \rr$ intersects $L_I$ in an open disk centered at the orthogonal projection $p_I$ of $p$ on $L_I$. Such centers form a compact and path-connected subset of $L_I$.
\end{itemize}
In both cases, $\Gamma_I$ is a domain in $L_I$ by the argument we already used for $\Gamma$.
\end{proof}

We are now in a position to prove the first result we announced.

\begin{theorem}[Local Extension]\label{localextension}
Let $f$ be a slice regular function on a slice domain $\Omega$. For every $p_0 \in \Omega$, there exist a symmetric slice domain $N$ with $N \cap \rr \subset \Omega$, a slice domain $\Lambda$ with $p_0 \in \Lambda \subseteq \Omega \cap N$, and a slice regular function $g : N \to \hh$ such that $g$ coincides with $f$ in $N \cap \rr$, whence in $\Lambda$.
\end{theorem}

\begin{proof}
If $p_0 \in \Omega \cap \rr$ then the thesis is obvious because $\Omega$ contains an open ball centered at $p_0$. We can therefore suppose $p_0=x_0+y_0J_0$ with $x_0,y_0\in\rr$, $y_0>0$ and $J_0\in\s$.

Since $\Omega$ is a slice domain, there exists a continuous path $\gamma : [0,1] \to \Omega_{J_0},  \gamma(t)= \alpha(t)+J_0\beta(t)$, such that $\gamma(0) = p_0$ and $\gamma(1) \in \rr$. Up to restricting and reparametrizing $\gamma$, we can suppose the support of $\gamma$ to only intersect the real axis at $\gamma(1)$. By Lemma~\ref{gamma}, there exists $\varepsilon>0$ such $M := \Gamma(\gamma([0,1]),\varepsilon)$ is a slice domain with the property $\gamma([0,1]) \subset M \subseteq \Omega$. Let $t_0\in[0,1]$ be such that $\max_{[0,1]}|\im(\gamma)|=|\im(\gamma(t_0))|$.

If we choose $K_0 \in \s$ with $0<|K_0-J_0|<\frac{\varepsilon}{|\im(\gamma(t_0))|}$, then the support of the curve $\alpha+K_0\beta$ is included in $M_{K_0}$. Indeed, the distance between $\alpha(t)+K_0\beta(t)$ and $\gamma(t)=\alpha(t)+J_0\beta(t)$ is $|K_0-J_0|\beta(t)$, which is less than $\frac{|\im(\gamma(t))|}{|\im(\gamma(t_0))|}\varepsilon$ for all $t\in[0,1)$ and is $0$ for $t=1$.

Let $N$ be the symmetric completion of the connected set $M_{K_0}$. We point out the following properties of $N$: it includes the support of $\gamma=\alpha+J_0\beta$; it has $N_{K_0}^+=M_{K_0}^+$; and it is a slice domain. Moreover, $N_{J_0}^+ \subseteq M_{J_0}^+ \subseteq \Omega_{J_0}$. Indeed, for each $x+yJ_0 \in N_{J_0}^+$ it holds $x+yK_0 \in N_{K_0}^+=M_{K_0}^+$. By direct computation, for all $t\in[0,1]$, it holds
\[|x+yK_0-\gamma(t)|^2-|x+yJ_0-\gamma(t)|^2 = 2y\beta(t) (1-\langle K_0, J_0 \rangle) \geq 0 \,.\]
Thus, the distance between $x+yJ_0$ and $\gamma(t)$ is less than, or equal to, the distance between $x+yK_0$ and $\gamma(t)$. If $B(\gamma(1),\varepsilon)$ includes $x+yK_0$, then it also includes $x+yJ_0$. Similarly, if there exists $t\in[0,1)$ such that $B\left(\gamma(t),\frac{|\im(\gamma(t))|}{|\im(\gamma(t_0))|}\varepsilon\right)$ includes $x+yK_0$, then the same ball includes $x+yJ_0$. In both cases, $x+yJ_0$ belongs to $M_{J_0}^+$.

By the General Extension Formula~\ref{extensionformulathm} there exists a unique slice regular function $g: N \to \hh$ that coincides with $f_{J_0}$ in $N_{J_0}^+ \subseteq M_{J_0}^+ \subseteq \Omega_{J_0}$, with $f_{K_0}$ in $N_{K_0}^+=M_{K_0}^+$ and with $f$ in $N \cap \rr = M_{K_0} \cap \rr$.

Within the open set $N \cap \Omega$, the slice $(N \cap \Omega)_{J_0}$ includes the support of $\gamma$. Lemma~\ref{gamma} guarantees that there exists $\varepsilon_0>0$ such that the slice domain $\Lambda := \Gamma(\gamma([0,1]),\varepsilon_0)$ has the property $\gamma([0,1]) \subset \Lambda \subseteq N\cap\Omega$. Now, $g$ and $f$ coincide in $N \cap \Omega \cap \rr = N \cap \rr$, whence throughout $\Lambda$ by the Identity Principle~\ref{identity}.
\end{proof}

We can draw several consequences. Since any slice regular function $g$ on a symmetric slice domain is real analytic by~\cite[Proposition 7]{perotti} and~\cite[Theorem 2.4]{global}, it follows that:

\begin{corollary}
Every slice regular function on a slice domain is real analytic.
\end{corollary}

The second consequence of Theorem~\ref{localextension} is the main result of this work.

\begin{theorem}[Local Representation Formula]\label{localrepresentationthm}
Let $\Omega$ be a slice domain and let $f : \Omega \to \hh$ be a slice regular function. For all $J,K \in \s$ with $J \neq K$ and all $x,y \in \rr$ with $y > 0$ such that $x+yJ, x+yK \in \Omega$, let us set
\begin{eqnarray*}
b(x+yJ,x+yK) &:=& (J-K)^{-1} \left[J f(x+yJ) - K f(x+yK)\right]\,,\\
c(x+yJ,x+yK) &:=& (J-K)^{-1} \left[f(x+yJ) - f(x+yK)\right] \,.
\end{eqnarray*}
Additionally, let us set $b(x,x) := f(x)$ and $c(x,x) := 0$ for all $x \in \Omega \cap \rr$. For every $p_0 \in \Omega$, there exists a slice domain $\Lambda$ with $p_0 \in \Lambda \subseteq \Omega$ such that the following properties hold for all $x,y \in \rr$ with $y \geq 0$:
\begin{itemize}
\item If $U = (x+y\s) \cap \Lambda$ is not empty, then $b,c$ are constant in $(U \times U) \setminus \{ (u,u) : u \in U \setminus \rr\}$.
\item If $I,J,K \in \s$ with $J \neq K$ are such that $x+yI,x+yJ,x+yK \in \Lambda$, then
\begin{equation}
f(x+yI) = b(x+yJ,x+yK) + I c(x+yJ,x+yK) \,. \label{localrepresentationformula}
\end{equation}
\end{itemize}
\end{theorem}

\begin{proof}
By Theorem~\ref{localextension}, there exist a symmetric slice domain $N$, a slice domain $\Lambda$ with $p_0 \in \Lambda \subseteq \Omega \cap N$ and a slice regular function $g : N \to \hh$ such that $g$ coincides with $f$ in $\Lambda$. If we apply Theorem~\ref{R-representationformula} to $g:N\to\hh$ at $x+yJ,x+yK \in N$ (with $x,y \in \rr$ and $y \geq 0$), we can conclude that the quaternions $(J-K)^{-1} \left[J g(x+yJ) - K g(x+yK)\right]$ and $(J-K)^{-1} \left[g(x+yJ) - g(x+yK)\right]$ do not depend on the choice of $J,K\in\s$ with $J \neq K$. When $x+yJ,x+yK \in \Lambda$, these quaternions coincide with $b(x+yJ,x+yK)$ and $c(x+yJ,x+yK)$, respectively. Thus, the quaternions $b(x+yJ,x+yK),c(x+yJ,x+yK)$ do not depend on the choice of $x+yJ,x+yK$ within $U := (x+y\s)\cap\Lambda$. By construction,
\[f(x+yI) = g(x+yI) = b(x+yJ,x+yK) + I c(x+yJ,x+yK) \]
when $x+yI,x+yJ,x+yK \in \Lambda$.
\end{proof}


\section{The Extension Theorem on simple domains}\label{sec:extension}

This section proves the Extension Theorem for simple slice domains announced in the Introduction. We recall that a slice domain $\Omega$ is called simple if, for any choice of $J,K \in \s$, the set
\[\Omega_{J,K}^+:=\{x+yJ \in \Omega_{J}^+: x+yK \in \Omega_{K}^+\}\]
is connected.

The next definition, proposition and remark provide many examples of simple domains.

\begin{definition}
A set $T\subseteq\hh$ is \emph{slice convex} if, for every $I\in\s$, the slice $T_I$ is a convex subset of $L_I$.
\end{definition}

\begin{proposition}
Let $\Omega$ be a slice domain. If $\Omega$ is slice convex, then it is a simple domain.
\end{proposition}

\begin{proof}
If $\Omega$ is slice convex then each half-slice $\Omega_J^+$ is convex because it is the intersection of the convex set $\Omega_J$ with the (convex) half-plane $L_J^+$. Moreover, for each $J,K \in \s$, the set
\[A := \{x+yJ \in L_J^+ : x+yK \in \Omega_K^+\}\]
is convex because it is a ``copy'' within the half-plane $L_J^+$ of the convex set $\Omega_K^+$. The set $\Omega_{J,K}^+$, which is the intersection of the convex sets $\Omega_J^+$ and $A$, is convex, whence connected.
\end{proof}

A similar technique proves what follows.

\begin{remark}
Let $\Omega$ be a slice domain. If $\Omega$ is starlike with respect to a point $x_0 \in \Omega \cap \rr$, then $\Omega$ is a simple domain.
\end{remark}

Here is an example of slice domain that is not simple.

\begin{example}\label{nonsimple}
The slice domain $\Omega$ of Example~\ref{ex:douren} is not a simple domain. Indeed, $\Omega_{I,-I}^+$ has two connected components: one is the open disk that forms the interior of $D_I$ in $L_I$; the other is $L_I^+\setminus(h_I\cup D_I)$.
\end{example}

For simple slice domains, the following result holds.

\begin{theorem}[Extension]\label{extension}
Let $f$ be a slice regular function on a simple slice domain $\Omega$. There exists a unique slice regular function $\tilde f : \widetilde{\Omega} \to \hh$ that extends $f$ to the symmetric completion of its domain.
\end{theorem}

\begin{proof}
For all $J \in \s$, let us adopt the notation $N(J)$ for the unique symmetric set such that $N(J)_J^+=\Omega_J^+$ and such that $N(J)\cap\rr=\Omega\cap\rr$. For all $J,K \in \s$, the intersection $N(J,K) := N(J) \cap N(K)$ is a slice domain because $\Omega$ is simple. We divide our proof into three steps.
\begin{enumerate}
\item Let us prove that, for each $J_0 \in \s$, there exists a slice regular function $g_0 : N(J_0) \to \hh$ that coincides with $f$ in $N(J_0)_{J_0}^+=\Omega_{J_0}^+$ and in $N(J_0)\cap\rr=\Omega\cap\rr$.
\begin{enumerate}
\item For each $K \in \s \setminus \{J_0\}$, let $g_0^K : N(J_0,K) \to \hh$ denote the slice regular function that coincides with $f_{J_0}$ in $N(J_0,K)_{J_0}^+ = \Omega_{J_0,K}^+$, with $f_K$ in $N(K,J_0)_{K}^+ = \Omega_{K,J_0}^+$ and with $f$ in $N(J_0,K) \cap \rr = \Omega \cap \rr$. Such a function exists by Theorem~\ref{extensionformulathm} and it is defined as $g_0^K(x+yI) := b(x+yJ_0,x+yK) + I c(x+yJ_0,x+yK)$ with
\[b(x+yJ_0,x+yK) := (J_0-K)^{-1} \left[J_0 f(x+yJ_0) - K f(x+yK)\right]\]
and
\[c(x+yJ_0,x+yK) := (J_0-K)^{-1} \left[f(x+yJ_0) - f(x+yK)\right] \]
for all $x,y \in \rr$ with $y \geq 0$ such that $x+y\s \subset N(J_0,K)$. For $y=0$, we get $b(x,x)=f(x)$ and $c(x,x)=0$.
\item Fix $p_0=x_0+y_0J_0$, either in $N(J_0)_{J_0}^+$ or in $N(J_0)\cap\rr$. By Theorem~\ref{localrepresentationthm}, $p_0$ has a neighborhood $\Lambda$ such that, for $\Lambda' := \Lambda \setminus \Lambda_{J_0}^+$, the maps $\Lambda' \ni x+yK \mapsto b(x+yJ_0,x+yK)$ and $\Lambda' \ni x+yK \mapsto c(x+yJ_0,x+yK)$ are constant with respect to $K$. Let us denote these constants as $b(x+yJ_0),c(x+yJ_0)$, respectively, and let us set 
\[g_0(x_0+y_0I) := b(x_0+y_0J_0) + I c(x_0+y_0J_0)\]
for all $I \in \s$. We have thus constructed a function $g_0 : N(J_0) \to \hh$.
\item For each $x_0+y_0\s \subset N(J_0)$ (with $x_0,y_0 \in \rr, y_0 \geq 0$), there exists a neighborhood $\Lambda$ of $x_0+y_0J_0$ such that the equality
\[ g_0(x+yI) = b(x+yJ_0,x+yK) + I c(x+yJ_0,x+yK) = g_0^K(x+yI) \]
holds for all $x,y \in \rr$ with $y \geq 0$ and all $I,K \in \s$ such that $x+yK \in \Lambda'$. For each $\delta>0$, let us consider the following neighborhood of $x_0+y_0\s$:
\[ T(x_0+y_0\s,\delta) := \{ x+yI : |x-x_0|^2+|y-y_0|^2<\delta^2, I \in \s \} \,.\]
There exist $\delta, \varepsilon > 0$ such that $\Lambda$ contains the set
\[ \{x+yK \in T(x_0+y_0\s,\delta) : |K-J_0|<\varepsilon\}\,.\]
If we pick any $K_0 \in \s$ such that $0<|K_0-J_0|<\varepsilon$, then $g_0$ coincides with $g_0^{K_0}$ in $T(x_0+y_0\s,\delta)$. It follows at once that $g_0$ is slice regular in $T(x_0+y_0\s,\delta)$ and that $g_0(x_0+y_0J_0)=f(x_0+y_0J_0)$, as desired.
\end{enumerate}
\item Any two slice regular functions $g_0 : N(J_0) \to \hh$ and $g_1 : N(J_1) \to \hh$ that coincide with $f$ in $N(J_0) \cap \rr = \Omega \cap \rr = N(J_1) \cap \rr$ coincide with each other in the slice domain $N(J_0,J_1)$ by the Identity Principle~\ref{identity}.
\item By steps {1.} and {2.}, there exists a slice regular function $g$ on
\[\bigcup_{J\in\s} N(J) = \widetilde \Omega\]
that coincides with $f$ in $\Omega \cap \rr$. By the Identity Principle~\ref{identity}, $g$ coincides with $f$ in the slice domain $\Omega$.
\end{enumerate}
\end{proof}




\end{document}